\documentclass[a4paper,12pt]{amsart}

\pdfoutput=1

\usepackage[T1]{fontenc}
\usepackage[english]{babel}

\usepackage[
a4paper ,
top = 3cm ,
bottom = 2.5cm ,
left = 2.5cm ,
right = 2.5cm
]{geometry}

\let\oldtocsubsection=\tocsubsection
\renewcommand{\tocsubsection}[2]{\hspace{1.9pc}\oldtocsubsection{#1}{#2}}

\usepackage{mathtools} 
\mathtoolsset{mathic} 
\usepackage[ scale = 1.09 ]{Baskervaldx} 
\usepackage{inconsolata} 
\usepackage{eulerpx} 
\usepackage{eucal} 
\usepackage{xcolor} 

\usepackage{bm}

\makeatletter
\g@addto@macro\bfseries{\boldmath}
\makeatother

\DeclareSymbolFont{myoperators}{T1}{AlegreyaSans-LF}{sb}{n}
\DeclareSymbolFontAlphabet{\mathsf}{myoperators}
\makeatletter
\renewcommand{\operator@font}{\mathgroup\symmyoperators}
\makeatother

\DeclareFontEncoding{LS1}{}{}
\DeclareFontSubstitution{LS1}{stix2}{m}{n}
\DeclareSymbolFont{stixaccent}{LS1}{stix2}{m}{it}
\DeclareMathAccent{\stixwidehat}{\mathord}{stixaccent}{"9A}
\DeclareMathSymbol{\stixmodels}{\mathrel}{stixaccent}{"F5}

\allowdisplaybreaks[1] 

\usepackage[normalem]{ulem}
\newcommand{\soutc}{\bgroup\markoverwith{\textcolor{cyan}{\rule[0.4ex]{1pt}{1.7pt}}}\ULon}

\usepackage{pifont}
\makeatletter
\renewcommand{\@fnsymbol}[1]{%
\ifcase#1\hbox{} \or \mbox{\ding{118}} \or \mbox{\ding{168}} \or \mbox{\ding{74}} \or \mbox{\ding{116}} \or \mbox{\ding{110}} \or \mbox{\ding{115}} \or \mbox{\ding{94}} \or \mbox{\ding{117}} \or \mbox{\ding{108}} \or \mbox{\ding{67}}
\else\@ctrerr\fi\relax}
\makeatother

\DeclareMathOperator{\lpow}{LPOW}
\DeclareMathOperator{\mpow}{MPOW}
\DeclareMathOperator{\ppow}{PPOW}
\DeclareMathOperator{\pow}{POW}
\DeclareMathOperator{\Char}{char}

\newcommand{\F}{ \varmathbb{F} }
\newcommand{\Q}{ \varmathbb{Q} }
\newcommand{\Z}{ \varmathbb{Z} }
\newcommand{\Zp}{ \varmathbb{Z}^{+} }

\newcommand{\CF}{ \mathcal{F} }
\newcommand{\CH}{ \mathcal{H} }
\newcommand{\CL}{ \mathcal{L} }

\newcommand{\FM}{ \mathfrak{M} }

\bmdefine{\CX}{ \mathcal{X} }
\bmdefine{\CZ}{ \mathcal{Z} }

\bmdefine{\Lor}{ \lor }
\bmdefine{\Land}{ \land }
\bmdefine{\To}{ \mathrel{ { \xrightarrow{ \hspace{3mm} } } } }
\bmdefine{\Lnot}{ \lnot }
\bmdefine{\Exists}{ \exists }
\bmdefine{\Forall}{ \forall }

\DeclareFontEncoding{LMS}{}{}
\DeclareFontSubstitution{LMS}{npxsy}{m}{n}
\DeclareSymbolFont{symbolsN}{LMS}{npxsy}{b}{n}

\DeclareFontEncoding{LMX}{}{}
\DeclareFontSubstitution{LMX}{npxexx}{m}{n}
\DeclareSymbolFont{largesymbolsN}{LMX}{npxexx}{b}{n}

\DeclareMathDelimiter{\Lparen}{\mathopen}{symbolsN}{185}{largesymbolsN}{0}
\DeclareMathDelimiter{\Rparen}{\mathclose}{symbolsN}{186}{largesymbolsN}{1}
\DeclareMathDelimiter{\Lbrack}{\mathopen}{symbolsN}{187}{largesymbolsN}{2}
\DeclareMathDelimiter{\Rbrack}{\mathclose}{symbolsN}{188}{largesymbolsN}{3}

\newcommand{\InterpEq}{ { =_{ \Gamma } } } 
\newcommand{\InterpLeq}{ { \leq_{ \Gamma } } } 
\newcommand{\InterpDiv}{ { \bigm|_{ \Gamma } } } 

\DeclareTextFontCommand{\noun}{\scshape} 

\numberwithin{equation}{section} 

\usepackage[pagebackref]{hyperref}
\hypersetup{
allcolors = blue ,
breaklinks = true ,
colorlinks = true ,
pdfencoding = unicode ,
pdfpagelayout = OneColumn ,
pdfstartview = {FitH} ,
}

\usepackage[bold]{bookmark} 

\usepackage[
alphabetic ,
msc-links ,
nobysame
]{amsrefs}

\makeatletter
\def\print@backrefs#1{\space\SentenceSpace[cited on page(s) \csname br@#1\endcsname]}
\makeatother

\usepackage{enumitem}
\setlist[enumerate,1]{ label = \texttt{\textup{\alph*.}} }

\usepackage{aliascnt} 

\usepackage{cleveref}
\makeatletter
\crefdefaultlabelformat{\leavevmode\sw@slant#2{\upshape#1}#3}
\makeatother
\creflabelformat{equation}{#2(\textup{#1})#3}

\usepackage[framemethod=tikz]{mdframed}
\mdfdefinestyle{barra}{
linewidth = 3pt ,
topline = false ,
bottomline = false ,
rightline = false ,
innertopmargin = -4pt ,
innerbottommargin = 2pt ,
innerrightmargin = 0pt ,
}

\theoremstyle{plain}
\newmdtheoremenv[ default , style = barra ]{theorem}[dummy]{Theorem}
\newmdtheoremenv[ default , style = barra ]{corollary}[dummy]{Corollary}
\newmdtheoremenv[ default , style = barra ]{proposition}[dummy]{Proposition}
\newmdtheoremenv[ default , style = barra ]{lemma}[dummy]{Lemma}

\theoremstyle{definition}
\newmdtheoremenv[ default , style = barra ]{definition}[dummy]{Definition}
\newmdtheoremenv[ default , style = barra ]{example}[dummy]{Example}
\newmdtheoremenv[ default , style = barra ]{remark}[dummy]{Remark}

\crefname{theorem}{theorem}{theorems}
\crefname{corollary}{corollary}{corollaries}
\crefname{proposition}{proposition}{propositions}
\crefname{lemma}{lemma}{lemmas}
\crefname{definition}{definition}{definitions}
\crefname{example}{example}{examples}
\crefname{remark}{remark}{remarks}

\newlist{enumtheorem}{enumerate}{1}
\setlist[enumtheorem]{
label = \texttt{\textup{\alph*.}} ,
ref = \thetheorem\texttt{\textup{\alph*}} ,
}
\crefalias{enumtheoremi}{theorem}

\newlist{enumproposition}{enumerate}{1}
\setlist[enumproposition]{
label = \texttt{\textup{\alph*.}} ,
ref = \theproposition\texttt{\textup{\alph*}} ,
}
\crefalias{enumpropositioni}{proposition}

\newlist{enumlemma}{enumerate}{1}
\setlist[enumlemma]{
label = \texttt{\textup{\alph*.}} ,
ref = \thelemma\texttt{\textup{\alph*}} ,
}
\crefalias{enumlemmai}{lemma}

\makeatletter
\newcounter{subcreftmpcnt}%
\newcommand\alphsubformat[1]{\texttt{\textup{\alph{#1}}}}
\newcommand\subcref[2][\alphsubformat]{%
\ifcsname r@#2@cref\endcsname
\cref@getcounter {#2}{\mylabel}%
\setcounter{subcreftmpcnt}{\mylabel}%
\hyperref[#2]{\alphsubformat{subcreftmpcnt}}%
\else ?? \fi}
\makeatother

\usepackage[doipre={}]{uri} 

\begin{document}

\title[Undecidability of indecomposable polynomial rings]{Remarks on undecidability of \\ indecomposable polynomial rings}

\author{Marco Barone}
\address{Departamento de Letras \\ Universidade Federal de Pernambuco \\ Avenida da Arquitetura, S/N - Cidade Universitária \\ Recife/PE - Brasil - 50740-550}
\email[M.~Barone]{marco.barone@ufpe.br}

\author{Nicolás Caro-Montoya}
\address{Departamento de Matemática \\ Universidade Federal de Pernambuco \\ Avenida Jornalista Aníbal Fernandes, S/N - Cidade Universitária \\ Recife/PE - Brasil - 50740-560}
\email[N.~Caro-Montoya]{jorge.caro@ufpe.br}

\author{Eudes Naziazeno}
\email[E.~Naziazeno]{eudes.naziazeno@ufpe.br}


\begin{abstract}

We prove that arithmetic is interpretable in any indecomposable polynomial ring (in any set of variables), and in addition we provide an alternative uniform proof of undecidability for all members in this class of rings.

\end{abstract}

\maketitle 

\tableofcontents

\section{Introduction}

\noindent The purpose of this note is twofold: first, after some preliminary definitions and results (\Cref{subsection-preliminaries}), we provide two interpretations of arithmetic in polynomial rings in one variable over a field.

The first interpretation (\Cref{subsection-interpretability-R[x]-R-field-char>0}) works uniformly in the case of positive characteristic; in particular, it does not depend on such value of the characteristic. The bulk of required results are consequence of the far more general treatment from \cite{BCN-main}; however, the construction in this particular case can be carried out in a direct and elementary way. Moreover, though this result is apparently well-known, we were not able to find a specific source containing a full proof of it, and so this presentation is desirable for future reference.

The second interpretation already works uniformly for polynomial rings (in any set of variables) over reduced indecomposable coefficient rings that are not fields. In order to extend its scope (\Cref{subsection-interpretability-R[x]-R-field-infinite-multiplicative}), we prove a technical result for polynomial rings in one variable over a field having elements of infinite multiplicative order; in the course of the proof we characterize such fields. In this way, the second interpretation also works for the referred class of coefficient fields (\Cref{theorem-big-interpretation}). Since this subclass includes the characteristic zero case, the two constructed interpretations deal with the entire class of univariate polynomial rings over a field.

Our second aim (\Cref{section-R[X]-indecomposable}) is to extend, to the larger class of indecomposable polynomial rings, our previous interpretability/undecidability results appearing in \cite{BCN-main}, valid only for \emph{reduced} indecomposable polynomial rings. To this end, we use the fact that the nilradical $N_{S}$ of a polynomial ring $S$ is definable and that $S / N_{S}$ is isomorphic to a reduced indecomposable polynomial ring whenever $S$ is indecomposable.

\subsection{Preliminaries} \label{subsection-preliminaries}

\noindent All the rings $S$ considered are commutative and unital. The set of units (invertible elements) of $S$ is denoted by $S^{*}$. An element $a\in S$ is said to be

\begin{itemize}[itemsep = 2mm]

\item \textbf{Nilpotent}, if $a^{n} = 0$ for some $n \geq 1$.

\item \textbf{Idempotent}, if $a^{2} = a$.

\item \textbf{Regular}, if $ab = ac$ with $b , c \in S$ implies $b = c$.

\item \textbf{Prime}, if $a \notin \{ 0 \} \cup S^{*}$, and whenever $a$ divides a product, it divides some of the factors.

\end{itemize}
The ring $S$ is said to be \textbf{reduced} if its only nilpotent element is zero, and \textbf{indecomposable} if its only idempotent elements are $0$ and $1$. We focus on reduced/indecomposable polynomial rings: for a set $\CX$ of indeterminates over a ring $R$, the polynomial ring $R[ \CX ]$ is reduced \textup{(}resp.~indecomposable\textup{)} if and only if $R$ is reduced \textup{(}resp.~indecomposable\textup{)} \cite{BCN20}*{Proposition 4.6}. Finally, the class of reduced indecomposable (polynomial) rings strictly includes that of (polynomial) integral domains \cite{BCN20}*{Section 4.2}.

Next, we introduce the two central notions of this note, namely: uniform interpretation of one structure in a class of structures, and uniform translation of sentences about one structure to sentences about structures in a class (see \cite{Hod93}*{Section 5.3} for a general discussion on these concepts). In what follows $\CL$ and $\CL'$ are signatures, $\FM$ is a $\CL$-structure with domain $M$, and $\FM'$ is an $\CL'$-structure with domain $M'$.
\vspace{4mm}
\begin{definition} \label{definition-interpretation}

Let $n$ be a positive integer. A \textbf{$\bm{n}$-dimensional interpretation $\bm{ \Gamma }$ of $\bm{ \FM' }$ in $\bm{ \FM }$} is defined to consist of the following:

\begin{itemize}[ itemsep = 3mm , topsep = 4mm ]

\item A $n$-variable formula $\partial_{ \Gamma }$ of signature $\CL$ (the \textbf{domain formula of} $\bm{ \Gamma }$).

\item A surjective map $F_{ \Gamma } \colon \Delta_{ \Gamma } \to M'$, where $\Delta_{ \Gamma }$ is the subset of $M^{n}$ defined by the formula $\partial_{ \Gamma }$ (the \textbf{coordinate map of} $\bm{ \Gamma }$).

\item A $2n$-variable formula $\InterpEq$ of signature $\CL$ that satisfies, for any $\bm{a} , \bm{b} \in \Delta_{ \Gamma }$, the equivalence
\[
\FM \stixmodels \bm{a} \InterpEq \bm{b} \iff \FM' \stixmodels F_{ \Gamma } ( \bm{a} ) = F_{ \Gamma } ( \bm{b} ) \, .
\]

\item For each $m \geq 1$ and each $m$-ary relation symbol $Q$ in $\CL'$, a $mn$-variable formula $Q_{ \Gamma }$ of signature $\CL$ that satisfies, for any $\bm{a}_{1} , \dotsc , \bm{a}_{m} \in \Delta_{ \Gamma }$, the equivalence
\[
\FM \stixmodels {Q_{ \Gamma }} ( \bm{a}_{1} , \dotsc , \bm{a}_{m} ) \iff \FM' \stixmodels Q \bigl( F_{ \Gamma }( \bm{a}_{1} ) , \dotsc , F_{ \Gamma }( \bm{a}_{m} ) \bigr) \, .
\]

\item For each $m \geq 0$ and each $m$-ary function symbol $G$ in $\CL'$ (so that $G$ is a constant symbol whenever $m = 0$), a $( m + 1 ) n$-variable formula $G_{ \Gamma }$ of signature $\CL$ that satisfies, for any $\bm{a}_{1} , \dotsc , \bm{a}_{m} , \bm{b} \in \Delta_{ \Gamma }$, the equivalence
\[
\FM \stixmodels {G_{ \Gamma }} ( \bm{a}_{1} , \dotsc , \bm{a}_{m} , \bm{b} ) \iff \FM' \stixmodels G \bigl( F_{ \Gamma }( \bm{a}_{1} ) , \dotsc , F_{ \Gamma }( \bm{a}_{m} ) \bigr) = F_{ \Gamma }( \bm{b} ) \, .
\]

\end{itemize}
If $\FM$ varies on a class $\CH$ of $\CL$-structures but the domain formula and the formulas $=_{ \Gamma } , Q_{ \Gamma }$ and $G_{ \Gamma }$ remain fixed, we say that the interpretation $\Gamma$ of $\FM'$ is \textbf{uniform} in $\CH$.
\end{definition}

\begin{definition} \label{definition-translation}

A \textbf{translation of $\bm{ \CL' }$-sentences about $\bm{ \FM' }$ to $\bm{ \CL }$-sentences about $\bm{ \FM }$} is a mapping that assigns, to each $\CL'$-sentence $\varphi$, a $\CL$-sentence $\stixwidehat{ \varphi }$ in such a way that $\FM \stixmodels \stixwidehat{ \varphi } \iff \FM' \stixmodels \varphi$. If such a mapping works for $\FM$ varying in a class $\CH$ of $\CL$-structures, we say that the translation of $\CL'$-sentences about $\FM'$ to $\CL$-sentences about the members of $\CH$ is \textbf{uniform}.

\end{definition}
\vspace{3mm}
\noindent For example, if $\FM , \FM'$ and $\Gamma$ are as in~\Cref{definition-interpretation}, then there is a natural translation $\stixwidehat{ \Gamma }$ of $\CL'$-sentences about $\FM'$ to $\CL$-sentences about $\FM$ \cite{Hod93}*{Theorem 5.3.2}. Moreover, such a translation is uniform whenever the interpretation $\Gamma$ is.

\begin{definition} \label{definition-induced-translation}

The translation $\stixwidehat{ \Gamma }$ is called the \textbf{translation of sentences induced by} $\bm{ \Gamma }$.

\end{definition}
\vspace{3mm}
\noindent Notice that a translation of $\CL'$-sentences about $\FM'$ to $\CL$-sentences about $\FM$ allows to conclude the undecidability of the full theory of $\FM$ from that of the full theory of $\FM'$. In addition, if this translation is uniform, then we can say that such a proof of undecidability is ``uniform'' across all $\CL$-structures in the class considered.

In our context we take $\FM'$ as the structure $( \Z , + , \times )$, the \textbf{(first-order) arithmetic}, which is known to be undecidable \cite{RobR51}*{p.~137}, and consider $\FM$ varying in a certain class of polynomial rings, which are also $( + , \times )$-structures. Some of the main results of this work, namely, \Cref{theorem-interpretation-R[x]-R-field-char-p,theorem-undecidability-R[X]-reduced-indecomposable}, deal with two auxiliary structures. These structures interpret arithmetic, as we show below. Since translation of sentences and interpretations can be composed (for the latter see \cite{Hod93}*{p.~218}), these results will be enough for our purposes.

\begin{proposition} \label{proposition-auxiliary-structures}

Arithmetic is interpretable in the structures \( ( \Zp , + , \mid \, ) \) and \( ( \Zp , \leq , \mid \, ) \). Consequently, the full theory of these structures is undecidable.

\end{proposition}

\begin{proof}

Notice that $1$ is definable in $( \Zp , \mid \, )$ by the formula $\Forall a \, \Lparen \, t \mid a \, \Rparen$. For any positive integers $j , k , m$ we have the equivalences \cite{RobR51}*{p.~146}
\vspace{1mm}
\begin{alignat*}{3}
\hspace{\leftmarginiv} & \bullet \hspace{\labelsep}     &     j & = k ( k + 1 )     &     \ & \iff \ j = \operatorname{lcm} ( k , k + 1 ) \, ; \\[2mm]
\hspace{\leftmarginiv} & \bullet \hspace{\labelsep}     &     j & = km              &     \ & \iff \ ( k + m ) ( k + m + 1 ) = k ( k + 1 ) + m ( m + 1 ) + 2j \, ,
\shortintertext{ hence }
\hspace{\leftmarginiv} & \bullet \hspace{\labelsep}     &     j & = km              &     \ & \iff \ \operatorname{lcm} ( k + m , k + m + 1 ) = \operatorname{lcm} ( k , k + 1 ) + \operatorname{lcm} ( m , m + 1 ) + 2j \, ,
\end{alignat*}

\noindent and the notion of least common multiple is clearly expressible in terms of divisibility. This shows that $( \Zp , + , \times )$ is interpretable in $( \Zp , + , \mid \, )$. On the other hand, addition and multiplication in $\Zp$ are first-order definable in terms of the successor operation and the relation of divisibility, by virtue of a result of \noun{Julia Robinson} \cite{RobJ49}*{Theorems 1.1 and 1.2}. Since the successor operation is first-order definable in terms of the usual order relation in $\Zp$, we conclude that $( \Zp , + , \times )$ is also interpretable in $( \Zp , \leq , \mid \, )$.

Finally, a two-dimensional interpretation of arithmetic in the structure $( \Zp , + , \times )$ can be given, via the well-known construction of the integers from the positive integers. By composing this interpretation with each one of the interpretations previously constructed we obtain the desired result.
\end{proof}

\noindent Our way to interpret positive integers in our rings of interest will be through sets of positive powers of suitable elements in these rings.

\begin{definition} \label{definition-LMPPOW-L}

For a ring $S$ and an element $\ell \in S$, we define the following sets:

\begin{itemize}[ itemsep = 2mm ]

\item $\pow ( \ell )$ is the set of positive powers of $\ell$.

\item $\lpow ( \ell )$ is the set consisting of elements $f \in S$ (also referred to in~\cite{BCN20} as \textbf{logical powers} of $\ell$) satisfying:

      \begin{itemize}[ itemsep = 1mm, topsep = 2mm ]
      
      \item $\ell$ divides $f$;
      
      \item $\ell - 1$ divides $f - 1$;
      
      \item Every divisor of $f$ is a unit or a multiple of $\ell$.
      
      \end{itemize}

\end{itemize}
In the particular case $\Char (S) > 0$ we define the following sets:

\begin{itemize}[ itemsep = 2mm , topsep = 2mm ]

\item $\mpow ( \ell )$ is the set of powers of $\ell$ with exponent a positive multiple of $\Char(S)$.

\item $\ppow ( \ell )$ is the set of powers of $\ell$ with exponent a positive power of $\Char(S)$.

\end{itemize}
In the particular case when $S = R[x]$, with $x$ an indeterminate over a field $R$,

\begin{itemize}[ topsep = 2mm ]

\item $L$ is the set of degree $1$ polynomials in $S$.

\end{itemize}

\end{definition}
\vspace{5mm}
\noindent Notice that the sets $\lpow( \ell )$ are first-order definable across all rings, using $\ell$ as parameter.
\vspace{3mm}
\begin{lemma} \label{lemma-POW-subset-LPOW}

Let \( S \) be a ring. If \( \ell \in S \) is regular and prime, then \( \pow ( \ell ) \subseteq \lpow ( \ell ) \).

\end{lemma}

\begin{proof}

Let $n \geq 1$. Obviously $\ell \mid \ell^{n}$ and $\ell-1 \mid \ell^{n} - 1$, and if $g$ is a divisor of $\ell^{n}$, say $\ell^{n} = gh$, then $\ell^{ n + 1 }$ cannot divide $h$ (otherwise we would have, by canceling, that $\ell$ divides $1$, which contradicts the primality of $\ell$). Thus, the largest $k$ with $\ell^{k}$ dividing $h$ must satisfy $k\leq n$. After canceling we get $\ell^{ n - k } = g \stixwidehat{h}$, with $\stixwidehat{h}$ not a multiple of $\ell$. If $k = n$, then $g$ is invertible; otherwise, $\ell$ divides $g \stixwidehat{h}$, so necessarily $\ell$ divides $g$ because $\ell$ is prime.
\end{proof}

\begin{proposition} \label{proposition-L-uniformly-definable}

The set \( L \) is uniformly definable across all the polynomial rings in one variable over a field, irrespective of its characteristic. For such rings we have \( \pow ( \ell ) = \lpow ( \ell ) \), for each \( \ell \in L \). In particular, the sets of positive powers of elements in \( L \) are uniformly definable with parameter \( \ell \in L \).
 
\end{proposition}

\begin{proof}

If $S = R[ \CX ]$, with $R$ a field and $\CX$ a set of indeterminates over $R$, then we have $R = \{ 0 \} \cup S^{*}$, and so $R$ is uniformly definable on these rings. With this first-order definition for ``$t \in R$'', consider the formula
\[
\alpha ( \ell ) \colon \ \ \ell \notin R \ \Land \ \Forall f \, \Exists q \, \Lparen \, f - q \ell \in R \, \Rparen \, .
\]
In the particular case $\CX = \{ x \}$ we have, by the division algorithm, that $\alpha ( \ell )$ holds true for every $\ell \in L$. Conversely, if $\ell \in R[x]$ satisfies $\alpha ( \, \cdot \, )$, then by taking $f = x$ we get $x - q \ell = r$, for some $q \in R[x]$ and $r \in R$. This implies that $q \ell = x - r$ has degree $1$, so necessarily $\deg( \ell ) = 1$ because $\ell \notin R$. This shows that the sets $L$ are uniformly definable.

On the other hand, if $\ell \in L$, then $\ell$ is prime and regular, hence $\pow ( \ell ) \subseteq \lpow ( \ell )$ by~\Cref{lemma-POW-subset-LPOW}. For the reverse inclusion, if $f \in \lpow ( \ell )$, then $\ell - 1 \mid f - 1$ forces to have $f \neq 0$ (because $\ell - 1 \in L$, hence $\ell - 1$ is noninvertible), and since $\ell$ divides $f$, it follows that $\deg (f) > 0$. Thus, by unique factorization we have $f = u \ell^{n}$, with $n \geq 1$ and $u$ coprime to $\ell$; in particular $u$ is not a multiple of $\ell$. Since $u$ divides $f$ and $f \in \lpow ( \ell )$, $u$ must be a unit, that is $u \in R[x]^{*} = R^{*}$. Finally, we have $f - 1 = u \ell^{n} - 1 = u \cdot ( \ell^{n} - 1 ) + u - 1$, and since $\ell - 1$ divides both $f - 1$ and $\ell^{n} - 1$, we conclude that $\ell - 1 \mid u - 1 \in R$, which implies $u - 1 = 0$, and this shows that $f = \ell^{n} \in \pow ( \ell )$.
\end{proof}

\section{Uniform interpretability of arithmetic in univariate polynomial rings over certain fields} \label{section-interpretability-R[x]-R-some-field}

\noindent In this section we provide two interpretations of arithmetic, each one working uniformly in univariate polynomial rings over fields $R$ satisfying, respectively, the conditions

\begin{itemize}

\item $\Char(R) > 0$, or

\item $R$ has elements of infinite multiplicative order.

\end{itemize}
These subclasses cover the entire class of univariate polynomial rings over a field, and they have nonempty intersection; see~\Cref{proposition-good-fields}. In what follows $R$ is a field and $x$ is an indeterminate over $R$.

\subsection{The case of coefficient fields of positive characteristic} \label{subsection-interpretability-R[x]-R-field-char>0}

\noindent Although the title of this subsection refers to fields of positive characteristic, the reader may notice that the hypotheses of~\Cref{lemma-MPOW-PPOW} below are independent of the characteristic (see the proof of~\Cref{corollary-PPOW-uniformly-definable}).

\begin{lemma} \label{lemma-MPOW-PPOW}

Let \( S \) be a ring of positive characteristic, and let \( \ell \in S \) be such that\textup{:}

\begin{itemize}[ itemsep = 2mm, topsep = 3mm ]

\item \( \ell - 1 \) is regular.

\item For any positive integers \( m \) and \( n \) we have \( \ell^{m} - 1 \mid \ell^{n} - 1 \) if and only if \( m \mid n \). \textup{(}In particular, all the positive powers of \( \ell \) are distinct.\textup{)}

\item \( 0_{S} \) is the only multiple of \( \ell - 1 \) in the set \( \CZ_{S} = \{ n \cdot 1_{S} \colon n \in \Z \} \).

\end{itemize}
For any \( y \in S \), we have the following\textup{:}

\begin{enumlemma}[ itemsep = 2mm , topsep = 3mm ]

\item \label{lemma-MPOW-PPOW-a} \( y \in \mpow ( \ell ) \) if and only if \( y \in \pow ( \ell ) \) and \( ( \ell - 1 )^{2} \mid y - 1 \).

\item \label{lemma-MPOW-PPOW-b} If in addition \( \Char (S) \) is a prime number, then \( y \in \ppow ( \ell ) \) if and only if

\begin{itemize}[ itemsep = 1mm, topsep = 2mm]

\item \( y \in \mpow ( \ell ) \), and

\item for any \( y' \in \pow ( \ell ) \) with \( y' \neq \ell \) and \( y' - 1 \mid y - 1 \), we have \( y' \in \mpow ( \ell ) \).

\end{itemize}

\end{enumlemma}

\end{lemma}

\begin{proof} \leavevmode

\begin{enumerate}

\item Given $n \in \Zp$, let $w_{n} ( \ell ) = 1 + \ell + \dotsb + \ell^{ n - 1 }$, so that $\ell^{n} - 1 = ( \ell - 1 ) w_{n} ( \ell )$. Moreover, by observing that
\[
w_{n} ( \ell ) =
\begin{cases}
1_{S} ,                                                                        & \textnormal{if} \ n = 1 ; \\
n \cdot 1_{S} + ( \ell - 1 ) \sum_{ k = 0 }^{ n - 2 } ( n - 1 - k ) \ell^{k} , & \textnormal{otherwise} ,
\end{cases}
\]
we conclude that $\ell - 1 \mid w_{n} ( \ell ) - n \cdot 1_{S}$. These facts, together with the hypotheses on the element $\ell$, allow us to argue that
\[
( \ell - 1 )^{2} \mid \ell^{n} - 1 \iff \ell - 1 \mid w_{n} ( \ell ) \iff \ell - 1 \mid n \cdot 1_{S} \iff n\cdot 1_{S} = 0_{S} ,
\]
and to conclude, for any $y = \ell^{n} \in \pow ( \ell )$, that $( \ell - 1 )^{2} \mid y - 1$ if and only if $\Char (S) \mid n$ in $\Zp$.

\item Let $p = \Char (S)$. Since $p$ is a prime number, the positive powers of $p$ can be characterized as those positive multiples of $p$ whose nontrivial divisors are already multiples of $p$. These facts, together with the hypotheses on $\ell$ and the result of item~\subcref{lemma-MPOW-PPOW-a}, yield the claim.\qedhere

\end{enumerate}

\end{proof}

\begin{corollary} \label{corollary-PPOW-uniformly-definable}

The sets \( \ppow ( \ell ) \) are uniformly definable with parameter \( \ell \in L \) in the class of polynomial rings in one variable over a field of positive characteristic. In particular, such a formula is independent of the characteristic.

\end{corollary}

\begin{proof}

We claim that, for any field $R$ (even with $\Char(R) = 0$), all elements $\ell \in L \subseteq R[x]$ satisfy the hypotheses of~\Cref{lemma-MPOW-PPOW}. This, together with~\Cref{proposition-L-uniformly-definable}, implies the result.

Obviously $\ell - 1$ is regular, and since in this case we have $\CZ_{S} = \{ n \cdot 1_{S} \colon n \in \Z \} \subseteq R$ and $\ell - 1$ divides no nonzero constant polynomial, it follows that $0_{S}$ is the only multiple of $\ell - 1$ in $\CZ_{S}$. Finally, let $m , n \in \Zp$ be such that $\ell^{m} - 1 \mid \ell^{n} - 1$, and write $n = qm + r$, with $q \geq 0$ and $0 \leq r < m$. Since
\[
\ell^{m} - 1 \mid \ell^{n} - 1 = \ell^{r} ( \ell^{ qm } - 1 ) + \ell^{r} - 1
\]
and $\ell^{m} - 1 \mid \ell^{ qm } - 1$, it follows that $\ell^{m} - 1 \mid \ell^{r} - 1$. If $r$ were nonzero, then we would have $m = \deg ( \ell^{m} - 1 ) \leq \deg ( \ell^{r} - 1 ) = r$, an absurd. Therefore $r=0$, that is $m \mid n$.
\end{proof}

\noindent The next results concern the uniform interpretation of equality of exponents that we intend to build.

\begin{lemma} \label{lemma-ell_1^k=uell_2^k}

Let \( R \) be a field of characteristic \( p > 0 \), and suppose that \( k = p^{m} \), with \( m \geq 1 \). Given \( \ell_{1} , \ell_{2} \in L \subseteq R[x] \), there exist \( u \in R^{*} \) and \( \rho \in R \) such that \( \ell_{1}^{k} = u \ell_{2}^{k} + \rho \).

\end{lemma}

\begin{proof}

We may write $\ell_{i} = a_{i} x + b_{i}$, with $a_{1} , a_{2} \neq 0$. By setting $s = b_{1} - \frac{ a_{1} b_{2} }{ a_{2} }$ and $v = \frac{ a_{1} }{ a_{2} }$, we get $\ell_{1} = v \ell_{2} + s$. Notice that raising elements of $R[x]$ to $k$th power is the $m$th iterate of the Frobenius endomorphism, hence an additive map. Therefore we may write $\ell_{1}^{k} = u \ell_{2}^{k} + \rho$, where $u = v^{k} \in R^{*}$ and $\rho = s^{k} \in R$.
\end{proof}

\begin{proposition} \label{proposition-interpretation-equality-char-p}

Let \( R \) be a field of characteristic \( p > 0 \) and let \( \ell_{1} , \ell_{2} \in L \subseteq R[x] \). Suppose that \( k_{1} , k_{2} \) are powers of \( p \) with positive exponents, and set \( y_{i} = \ell_{i}^{ k_{i} } \) \textup{(}so that \( y_{i} \in \ppow ( \ell_{i} ) \) for \( i = 1 , 2 \)\textup{)}. We have \( k_{1} = k_{2} \) if and only if there exist \( u \in R^{*} \) and \( \rho \in R \) such that \( y_{1} = u y_{2} + \rho \).

\end{proposition}

\begin{proof}

The ``only if'' part follows from~\Cref{lemma-ell_1^k=uell_2^k}. For the converse, just observe that, if two nonzero polynomials differ by a constant, then they must have the same degree, and that multiplying by a unit does not alter the degree.
\end{proof}

\noindent In what follows we interpret, in our rings of interest, the positive integers endowed with the usual order relation and the relation of divisibility. Incidentally, the formulas for interpretation of these two relations are similar.

\begin{theorem} \label{theorem-interpretation-R[x]-R-field-char-p}

There is a two-dimensional uniform interpretation \( \Gamma \) of \( ( \Zp , \leq , \mid \, ) \) in the class of polynomial rings in one variable over a field of positive characteristic. Consequently, arithmetic is uniformly interpretable in this class of rings.

\end{theorem}

\begin{proof}\hspace{-1mm}\footnote{This is a corrected version of the proof of \cite{BCN-main}*{Theorem 6.27} (which hopefully will be corrected in a future version of the preprint), where it is claimed that certain modified formulas interpret sum and divisibility. If $p > 0$ denotes the characteristic of the ring, then these modified formulas actually interpret, respectively, the ternary relation $p^{i} + p^{j} = p^{k}$ instead of $i + j = k$, and $p^{i} \mid p^{j}$ instead of $i \mid j$; we indeed use this second modified formula in order to interpret $\leq$.}
\unpenalty\unskip\unpenalty 
The second claim follows from~\Cref{proposition-auxiliary-structures}, so we concentrate on the intended interpretation. We define the domain formula of $\Gamma$ as
\[
\partial_{ \Gamma } (\ell , y ) \colon \ \ \ell \in L \ \Land \ y \in \ppow ( \ell ) \, .
\]
By~\Cref{proposition-L-uniformly-definable} and~\Cref{corollary-PPOW-uniformly-definable} this formula defines, in any ring $S$ of the form $R[x]$, with $R$ a field of characteristic $p > 0$, the set
\[
\Delta_{ \Gamma } = \bigl\{ \bigl( \ell , \ell^{ p^{n} } \bigr) \colon \ell \in L , n \in \Zp \bigr\} \subseteq S \times S \, .
\]
The coordinate map $F_{ \Gamma } \colon \Delta_{ \Gamma } \to \Zp$ is given by $F_{ \Gamma } \bigl( \ell , \ell^{ p^{n} } \bigr) = n$; notice that $F_{ \Gamma }$ is well-defined because the powers of elements in $L$ are all distinct (because $\deg ( \ell^{j} ) = j$ for all $j \in \Zp$), and obviously $F_{ \Gamma }$ is surjective.

The interpretation of equality is given by the four-variable formula
\begin{equation} \label{formula-interpretation-equality-char-p}
( \ell_{1} , y_{1} ) \InterpEq ( \ell_{2} , y_{2} ) \colon \ \ \Exists u \, \Exists \rho \, \Lparen \, u \in S^{*} \ \Land \ \rho \in \{ 0 \} \cup S^{*} \ \Land \ y_{1} = u y_{2} + \rho \, \Rparen \, . \tag{$\maltese$}
\end{equation}
Since $S^{*} = R^{*} = R \smallsetminus \{ 0 \}$, equivalence between $\InterpEq$ and the equality of positive integers follows from~\Cref{proposition-interpretation-equality-char-p}.

Consider the four-variable formula
\begin{alignat*}{3}
( \ell_{1} , y_{1} ) \InterpLeq ( \ell_{2} , y_{2} ) & \colon \ \     &           \Exists z' \, \Lbrack & \, z' \in \ppow ( \ell_{1} ) \ \Land \ ( \ell_{1} , z' ) \InterpEq ( \ell_{2} , y_{2} ) \ \Land \ y_{1} - 1                          &     & \mid z' - 1 \, \Rbrack \, ,
\intertext{ and the following slight variant: }
( \ell_{1} , y_{1} ) \InterpDiv ( \ell_{2} , y_{2} ) & \colon \ \     &     \Exists z' \, \biggl\Lbrack & \, z' \in \ppow ( \ell_{1} ) \ \Land \ ( \ell_{1} , z' ) \InterpEq ( \ell_{2} , y_{2} ) \ \Land \ \frac{ y_{1} }{ \ell_{1} } - 1     &     & \Bigm| \frac{ z' }{ \ell_{1} } - 1 \, \biggr\Rbrack \, .
\end{alignat*}
In both cases the subformula $\InterpEq$ is that given by~\labelcref{formula-interpretation-equality-char-p}.

If $y_{1} , z' \in \ppow ( \ell_{1} )$, then in particular $y_{1}$ and $z'$ are multiple of $\ell_{1}$. Therefore the quotients $y_{1} / \ell_{1}$ and $z' / \ell_{1}$ are well-defined because $R[x]$ is an integral domain, which justifies the shorthand used in the formula for the relation $\InterpDiv$. Given $( \ell_{1} , y_{1} ) , ( \ell_{2} , y_{2} ) \in \Delta_{ \Gamma }$, we have $y_{1} = \ell_{1}^{ p^{m} }$ and $y_{2} = \ell_{2}^{ p^{n} }$ for some $m , n \in \Zp$. We want to prove that
\makeatletter
\setbool{@fleqn}{true}
\makeatother
\vspace{1mm}
\begin{alignat*}{3}
\hspace{\leftmarginiv} & \bullet \hspace{\labelsep}     &     ( \ell_{1} , y_{1} ) & \InterpLeq ( \ell_{2} , y_{2} )     &     & \iff m \leq n \, , \ \textnormal{and} \\[1mm]
\hspace{\leftmarginiv} & \bullet \hspace{\labelsep}     &     ( \ell_{1} , y_{1} ) & \InterpDiv ( \ell_{2} , y_{2} )     &     & \iff m \mid n \, . \\[1mm]
\intertext{Condition $z' \in \ppow ( \ell_{1} )$ in formulas $\InterpLeq$ and $\InterpDiv$ amounts to $z' = \ell_{1}^{ p^{k} }$, with $k \in \Zp$, and condition $( \ell_{1} , z' ) \InterpEq ( \ell_{2} , y_{2} )$ amounts to $k = n$. Therefore we have} \\[-5mm]
\hspace{\leftmarginiv} & \bullet \hspace{\labelsep}     &     ( \ell_{1} , y_{1} ) & \InterpLeq ( \ell_{2} , y_{2} )     &     & \iff y_{1} - 1 \mid z' - 1 \iff \ell_{1}^{ p^{m} } - 1 \bigm| \ell_{1}^{ p^{n} } - 1 \, , \ \textnormal{and} \\[1mm]
\hspace{\leftmarginiv} & \bullet \hspace{\labelsep}     &     ( \ell_{1} , y_{1} ) & \InterpDiv ( \ell_{2} , y_{2} )     &     & \iff \frac{ y_{1} }{ \ell_{1} } - 1 \Bigm| \frac{ z' }{ \ell_{1} } - 1 \iff \ell_{1}^{ p^{m} - 1 } - 1 \bigm| \ell_{1}^{ p^{n} - 1 } - 1 \, . \\[1mm]
\intertext{Recall that for any $r , s \in \Zp$ and any $\ell \in L$ we have that $\ell^{r} - 1 \mid \ell^{s} - 1$ if and only if $r \mid s$ (see the proof of~\Cref{corollary-PPOW-uniformly-definable}). Consequently, the equivalences above become} \\[-5mm]
\hspace{\leftmarginiv} & \bullet \hspace{\labelsep}     &     ( \ell_{1} , y_{1} ) & \InterpLeq ( \ell_{2} , y_{2} )     &     & \iff p^{m} \mid p^{n} \iff m \leq n \, , \ \textnormal{and} \\[1mm]
\hspace{\leftmarginiv} & \bullet \hspace{\labelsep}     &     ( \ell_{1} , y_{1} ) & \InterpDiv ( \ell_{2} , y_{2} )     &     & \iff p^{m} - 1 \mid p^{n} - 1 \, .
\end{alignat*}
\vspace{3mm}
\noindent Finally, it is easy to see that $p^{m} - 1 \mid p^{n} - 1$ if and only if $m \mid n$, for any positive integer $p > 1$ (not just $p$ a prime). This proves that the formulas $\InterpLeq$ and $\InterpDiv$ correctly interpret $\leq$ and $\mid$ in $\Zp$, respectively.
\end{proof}

\makeatletter
\setbool{@fleqn}{false}
\makeatother

\subsection{The case of coefficient fields containing elements of infinite multiplicative order} \label{subsection-interpretability-R[x]-R-field-infinite-multiplicative}

In this subsection we prove that univariate polynomial rings over certain fields satisfy a technical property, which allows to interpret arithmetic uniformly on these rings; actually, such an interpretation already works uniformly in another wide class of polynomial rings (\Cref{theorem-big-interpretation}).

\begin{definition}

Let $S$ be a ring, and let $\ell , \ell' \subseteq S$. We say that $\ell$ and $\ell'$ are \textbf{$\bm{1}$-connected} (and write $\ell \sim \ell'$) when, taking any $y = \ell^{m} \in \pow ( \ell )$ and $y' = ( \ell' )^{n} \in \pow ( \ell' )$ such that $\ell - \ell' \mid y - y'$, we must have $m = n$. We say that $\ell$ and $\ell'$ are \textbf{$\bm{k}$-connected} (and write $\ell \sim_{k} \ell'$) if there exist $l_{0} , \dotsc , l_{k} \in S$ with $l_{0} = \ell$ and $l_{k} = \ell'$, such that $l_{ i - 1 } \sim l_{i}$ for $i = 1 , \dotsc , k$. If in addition the elements $l_{1} , \dotsc , l_{ k - 1 }$ can be chosen from a subset $T$ of $S$, we say that $\ell$ and $\ell'$ are $k$-connected \textbf{through} $\bm{T}$.

\end{definition}
The following result characterizes $1$-connectivity in univariate polynomial rings over a field.

\begin{lemma} \label{lemma-connectivity-fields}

Let \( \ell , \ell' \in L \subseteq R[x] \), and write \( \ell' = r \ell - s \), with \( r \in R^{*} \) and \( s \in R \). We have \( \ell' \sim \ell \) precisely when \textup{(}exactly\textup{)} one of the following conditions hold\textup{:}

\begin{itemize}

\item \( r = 1 \) and \( s = 0 \) \textup{(}that is \( \ell' = \ell \)\textup{)}, or

\item \( r \neq 1, s \neq 0 \), and \( \frac{s}{ r - 1 } \) has infinite multiplicative order.

\end{itemize}
    
\end{lemma}

\begin{proof}

Since the connectivity relation is defined in terms of the ring operations, it follows that such relation is invariant under ring isomorphisms. Therefore, by considering the (unique) ring automorphism of $R[x]$ acting as the identity on $R$ and sending $x$ to $\ell$, we conclude that $\ell' = r \ell - s \sim \ell \iff rx - s \sim x$. Thus, in order to prove the result, it is sufficient to consider the case $\ell = x$, so that $\ell' = rx - s$.

Given $m , n \in \Zp$ we have $x^{m} - x^{n} = [ ( \ell' )^{m} - x^{n} ] - [ ( \ell' )^{m} - x^{m} ]$ and $\ell' - x \mid ( \ell' )^{m} - x^{m}$, hence
\[
\ell' - x \mid ( \ell' )^{m} - x^{n} \iff \ell' - x = ( r - 1 )x - s \mid x^{m} - x^{n} .
\]
If $r = 1$, then $\ell' - x = -s$, which divides every element of $R[x]$ unless $s = 0$, so in this case $\ell' \sim x$ precisely when $s = 0$. If $r \neq 1$, then
\[
\ell' - x \mid x^{m} - x^{n} \iff x - \frac{s}{ r - 1 } \Bigm| x^{m} - x^{n} \iff \Bigl( \frac{s}{ r - 1 } \Bigr)^{m} = \Bigl( \frac{s}{ r - 1 } \Bigr)^{n} .
\]
Consequently, $rx - s \sim x$ if and only if all the positive powers of $\frac{s}{ r - 1 }$ are distinct, and it is straightforward to show that this in turn results in $s \neq 0$ and $\frac{s}{ r - 1 }$ to have infinite multiplicative order.
\end{proof}

\noindent The previous result naturally motive us to characterize those fields having nonzero elements with infinite multiplicative order, that is, not being roots of unity. For $p \in \Zp$ a prime, let $\F_{p} = \Z / p \Z$, the finite field with $p$ elements. Notice that every field of characteristic $p$ contains a copy of $\F_{p}$.

\begin{proposition} \label{proposition-good-fields}

The field \( R \) has no nonzero element with infinite multiplicative order if and only if \( \Char (R) = p > 0 \) and the field extension \( R / \F_{p} \) is algebraic.

\end{proposition}

\begin{proof}

If every $a \in R^{*}$ is a root of unit, then $\Char (R) = p > 0$ (otherwise $R$ would contain a copy of $\Q$, and $2 \in \Q \subseteq R$ would have infinite multiplicative order). Since $a$ is a root of some polynomial of the form $T^{m} - 1 \in \F_{p} [T]$ (with $m \in \Zp$), we conclude that $a$ is algebraic over $\F_{p}$, hence $R / \F_{p}$ is algebraic. Conversely, if $\Char (R) = p > 0$ and $R / \F_{p}$ is algebraic, then for every $a \in R^{*}$ the field extension $\F_{p} (a) / \F_{p}$ is algebraic and finitely generated, hence a finite field extension, and consequently $\F_{p} (a)$ is a finite field; in particular $a$ (as every nonzero element in $\F_{p} (a)$) has finite multiplicative order.
\end{proof}

\noindent The following result constitutes an extension of~\cite{BCN-main}*{Theorem 6.18} to our rings of interest:

\begin{theorem} \label{theorem-4-connectivity-fields-infinite-order}

If \( R \) has elements with infinite multiplicative order, then all elements of \( L \subseteq R[x] \) are \( 2 \)-connected with the element \( x \) through \( L \), and consequently \( 4 \)-connected with each other through \( L \).
    
\end{theorem}

\begin{proof}

Let $\ell' = ax - b \in L$, with $a \in R^{*}$ and $b \in R$. It is sufficient to find elements $r, r',s, s' \in R$ such that $\ell' = r' \ell - s'$ and $\ell = rx - s$, and satisfying the second set of conditions given by~\Cref{lemma-connectivity-fields}, namely:

\begin{itemize}

\item $r, r' \notin \{ 0 , 1 \}$;

\item $s, s' \neq 0$;

\item the elements $\frac{s}{ r - 1 }$ and $\frac{s'}{ r' - 1 }$ are not roots of unity.

\end{itemize}
In fact, in such a case we have $\ell \in L$, and the connections $\ell' \sim \ell$ and $\ell \sim x$ follow from~\Cref{lemma-connectivity-fields}.

Since $\ell' = ax - b = r'( rx - s ) - s'$, we get $a = rr'$ and $b = r's + s'$. Therefore $r' = ar^{-1}$ and $s' = b - ar^{-1}s$, and so the intended conditions on $r, r', s, s'$ amount to the existence of elements $r, s \in R$ satisfying:

\begin{itemize}

\item $r \neq 0 , 1 , a$;

\item $s, as - br \neq 0$;

\item the elements $\frac{s}{ r - 1 }$ and $\frac{ -rs' }{ -r( r' - 1 ) } = \frac{ as - br }{ r - a }$ are not roots of unity.
    
\end{itemize}
Let $k$ be the prime subfield of $R$, that is $k = \Q$ if $\Char (R) = 0$ and $k = \F_{p}$ if $\Char (R) = p > 0$. Let $K = k( a , b ) \subseteq R$. If $K$ has no nonzero element with infinite multiplicative order, then by~\Cref{proposition-good-fields} we have $\Char (K) = p > 0$ and the field extension $K / \F_{p} = K / k$ is algebraic. Since $R$ contains elements of infinite multiplicative order and $\Char (R) = p > 0$, it follows from~\Cref{proposition-good-fields} that the field extension $R / k$ cannot be algebraic, which implies that $R / K$ is necessarily a transcendental extension.
If $t \in R$ is any transcendental element over $K$, then we may take $r = t$ and $s = t^{2}$. In fact, by recalling that transcendental elements behave as indeterminates, and taking into account that $a , b \in K$ by definition, we conclude that $r = t \notin K$, hence $r \neq 0 , 1 , a$, that $s = t^{2} \neq 0$ and $as - br = at^{2} - bt \neq 0$, and that $\frac{s}{ r - 1 } = \frac{ t^{2} }{ t - 1 }$ and $\frac{ as - br }{ r - a } =  \frac{ at^{2} - bt }{t - a}$ behave as nonconstant rational functions in $t$ with coefficients in $K$, so in particular both have infinite multiplicative order.

It remains to consider the case in which $K = k( a , b )$ has elements with infinite multiplicative order. From basic field theory, we must consider the following three possibilities:

\begin{itemize}

\item The element $a$ is transcendental over the subfield $k(b)$.

\item The element $b$ is transcendental over the subfield $k(a)$.

\item The elements $a$ and $b$ are algebraic over $k$.

\end{itemize}
In the first case, by using that $a$ behaves as an indeterminate over $k(b)$, it is readily verified that the choices $r = s = a^{2}$ work (in this case $as - br = a^{3} - b a^{2}, \frac{s}{ r - 1 } = \frac{ a^{2} }{ a^{2} - 1 }$ and $\frac{ as - br }{ r - a } = \frac{ a^{2} - ba }{a - 1 }$). Similarly, if $b$ is transcendental over $k(a)$, then we may take $r = s = b$. Finally, if $a$ and $b$ are algebraic over $k$, then $K / k$ is algebraic, so by~\Cref{proposition-good-fields} we cannot have $\Char (R) > 0$, hence $k = \Q$, and since $K$ is finitely generated, it follows that $K$ is a number field (a finite extension of $\Q$).
Any number field contains only finitely many roots of unity\footnote{If $E$ is a number field and $u \in E$ is a root of unity, say a $n$th primitive root, then $u$ has degree $\varphi (n)$ over $\Q$, where $\varphi$ is Euler's totient function \cite{Hun80}*{Proposition V.8.3}. If $n = q_{1}^{ m_{1} } \dotsm q_{j}^{ m_{j} }$ is the prime factorization of $n$, then for $i = 1 , \dotsc , j$ we have
\[q_{i} - 1 , 2^{ m_{i} - 1 } \leq ( q_{i} - 1 ) q_{i}^{ m_{i} - 1 } \leq \varphi (n) = [ \Q(u) : \Q ] \leq [ E : \Q ] \, .
\]
Therefore $q_{i} \leq 1 + [E : \Q ]$ and $m_{i} \leq 1 + \log_{2} [ E : \Q ]$, and so the prime factors of $n$, as well as their multiplicities, are bounded above by a constant depending only on $[ E : \Q ]$. This shows that the value $n$ can vary only on a finite set, and each such value can contribute with at most $\varphi (n)$ (primitive) roots of unity.}\!\!. Consequently, the condition on the elements $\frac{s}{ r - 1 }$ and $\frac{ as - br }{ r - a }$ can be rewritten as the system of inequations $s - u\cdot ( r - 1 ) \neq 0$ and $as - br - u\cdot ( r - a ) \neq 0$, where $u$ ranges over the finite set of roots of unity in $K$. These, together with the remaining required conditions on $r$ and $s$, amount to a finite set of polynomial (in fact, linear) inequations in the two unknowns $r, s$ over the field $K$. Since $K$ is infinite, such a system has always a solution\footnotemark\!\!.
\footnotetext{If $f_{1} , \dotsc , f_{j}$ are nonconstant polynomials in $m$ indeterminates over a field $E$ and $f = f_{1} \dotsm f_{j}$, then for any $\bm{a} \in E^{m}$ we have $f_{i}( \bm{a} ) = 0$ for some $i \iff f( \bm{a} ) = 0$, and the latter condition cannot hold for every $\bm{a} \in E^{m}$ whenever $E$ is infinite~\cite{AK21}*{Exercise (2.43)}.}
\end{proof}

\noindent We end this subsection with an extension, to our rings of interest, of the following interpretability result:

\begin{proposition}[\cite{BCN-main}*{Theorem 6.20}] \label{proposition-interpretability-RIP-nonfield}

There is a two-dimensional uniform interpretation of the structure \( ( \Zp , + , \mid \, ) \) in the class of polynomial rings \textup{(}in any set of indeterminates\textup{)} whose coefficient rings are reduced indecomposable nonfields.
    
\end{proposition}
The construction of such interpretation depends only on the following facts: there is a formula in the language of rings that defines, in any ring $S$, a subset $W$. In the case $S$ is a reduced indecomposable polynomial ring, such a subset $W$ is nonempty, and every element $\ell \in W$ satisfies the properties

\begin{itemize}

\item $\lpow ( \ell ) = \pow ( \ell )$,

\item all the powers of $\ell$ are distinct, and

\item for any $m , n \in \Zp$, we have $\ell^{m} - 1 \mid \ell^{n} - 1$ if and only if $m \mid n$.

\end{itemize}
If in addition the coefficient ring of the polynomial ring $S$ is a (reduced indecomposable) nonfield, then any pair of elements in $W$ is $4$-connected through $W$.

If $S$ is a univariate polynomial over a field, then $W = L$ \cite{BCN-main}*{Proposition 6.13\texttt{c}}. \Cref{proposition-L-uniformly-definable}, together with the proof of~\Cref{corollary-PPOW-uniformly-definable}, shows that in this case all elements $\ell \in L = W$ satisfy the bulleted conditions above. If in addition the coefficient field has nonroots of unity, then the $4$-connectivity condition holds by~\Cref{theorem-4-connectivity-fields-infinite-order}. These facts, together with~\Cref{proposition-auxiliary-structures}, imply the following strenghtening of~\Cref{proposition-interpretability-RIP-nonfield}:

\begin{theorem} \label{theorem-big-interpretation}

Arithmetic is uniformly interpretable in the class of polynomial rings satisfying \textup{(}exactly\textup{)} one of the following conditions\textup{:}

\begin{itemize}

\item The coefficient ring is a reduced indecomposable ring that is not a field.

\item The polynomial ring is univariate over a field having elements of infinite multiplicative order.
    
\end{itemize}
 
\end{theorem}

\begin{remark}

The conclusion in~\Cref{theorem-4-connectivity-fields-infinite-order} actually characterizes fields having nonroots of unity: in fact, if $R$ is a field such that every pair of elements in $L \subseteq R[x]$ are $4$-connected through $L$, then there exists $l_{0} , \dotsc , l_{4} \in L$ such that $l_{0} = x, l_{4} = x + 1$ and $l_{ i - 1 } \sim l_{i}$ for $i = 1 , \dotsc , 4$. In particular, for some $i$ we have $l_{ i - 1 } \neq l_{i}$, and so~\Cref{theorem-4-connectivity-fields-infinite-order} implies that $l_{i} = r l_{ i - 1 } - s$, with $r , s \in R^{*}$ such that $r \neq 1$ and the element $\frac{s}{ r - 1 } \in R^{*}$ has infinite multiplicative order.

\end{remark}

\section{Interpretability of arithmetic and uniform undecidability in the case of indecomposable polynomial rings} \label{section-R[X]-indecomposable}

\noindent In this section we first consider the class $\CF$ of reduced indecomposable polynomial rings, together with the following three subclasses:

\begin{itemize}

\item $\CF_{1}$, consisted of rings of the form $B[ \CX ]$, with $B$ a reduced indecomposable ring that is not a field and $\CX$ a set of indeterminates over $B$;

\item $\CF_{2}$, consisted of polynomial rings (in any set of indeterminates) over a reduced indecomposable ring of characteristic zero, and

\item $\CF_{3}$, consisted of polynomial rings \emph{in one variable} over a field of positive characteristic.

\end{itemize}
If $S = B[ \CX ] \in \CF$ is a polynomial ring in more than one indeterminate and $x \in \CX$, then we may rewrite $S = R[x]$, where $R = B[ \CX \smallsetminus \{ x \} ]$ is reduced and indecomposable. Moreover, $R$ is a \emph{polynomial} ring, hence a nonfield, and $x$ is an indeterminate over $R$. Consequently we have $S \in \CF_{1}$.

The previous argument shows that $\CF = \CF_{1} \cup \CF_{2} \cup \CF_{3}$, that the subclasses $\CF_{1}$ and $\CF_{2}$ overlap, and explains why we may define $\CF_{3}$ in terms of \emph{univariate} polynomials only. Notice that we also have $\CF = \CF_{1} \cup \CF_{2}' \cup \CF_{3}$, where $\CF_{2}'$ is the subclass of $\CF_{2}$ of polynomial rings (in any set of indeterminates) over a field of characteristic zero.

First, we state results about interpretations of arithmetic that are uniform on the subclasses of $\CF$ discussed above. In addition, we also state a result about uniform translation of sentences about arithmetic to sentences about reduced indecomposable polynomial rings (see~\Cref{definition-translation}), which provides a uniform proof of undecidability for all such rings; notice that the interpretability results just mentioned provide only partial uniform undecidability proofs in the respective subclasses covered. Finally, we make use of further algebraic properties of indecomposable/polynomial rings, in order to drop the reducedness hypothesis in these results.

\subsection{The results for reduced indecomposable polynomial rings}

\noindent First, we state a uniform definability result for the subclass $\CF_{2}'$.
\vspace{2mm}
\begin{theorem}

There is a formula \( \gamma \) that defines, in any polynomial ring \( S \) over a field of characteristic zero, the subset \( \varmathbb{N} \) of natural numbers, namely,
\vspace{2mm}
\begin{align*}
\gamma (t) \colon \ \Exists f \, \Exists g \, \bigl\Lbrack \, & f \notin \{ 0 \} \cup S^{*} \ \Land \ g \neq 0 \ \Land \ f \mid g \\
                                                              & \negmedspace \Land \ \Forall h \, \bigl\Lparen \, \Lbrack \, h \in \{ 0 \} \cup S^{*} \ \Land \ ( f + h ) \mid g \, \Rbrack \ {\To} \ \Lbrack \, ( f + h + 1 ) \mid g \ \Lor \ h = t \, \Rbrack \, \bigr\Rparen \, \bigr\Rbrack \, . 
\end{align*}
\vspace{2mm}
\noindent Consequently, arithmetic is uniformly interpretable in this class of rings.

\end{theorem}
\vspace{3mm}
\noindent The proof is straightforward; see \cite{RobR51}*{pp.~141,142}. See also the far more general result \cite{BCN20}*{Theorem 6.11}, which provides a formula that defines, in any reduced indecomposable polynomial ring $S$, the set of elements $n \cdot 1_{S}$ with $n \in \Zp$. Consequently $\varmathbb{N}$ is uniformly definable in the larger subclass $\CF_{2}$.

On the other hand, it is clear that even this stronger result does not imply interpretability of arithmetic in the case of positive characteristic, because in such a case the definable subset obtained is that of the integers modulo the characteristic. Fortunately,~\Cref{proposition-auxiliary-structures} together with~\Cref{theorem-interpretation-R[x]-R-field-char-p} deals with the uniform interpretation of arithmetic for the subclass $\CF_{3}$; the subclass $\CF_{1}$ is addressed, similarly, by using~\Cref{proposition-auxiliary-structures} and~\Cref{proposition-interpretability-RIP-nonfield}\footnotemark\!\!. Putting together these results, we obtain:

\footnotetext{Notice that this interpretation actually works uniformly on the subclass $\CF_{1} \cup \CF_{4}$, where $\CF_{4}$ is the subclass of univariate polynomial rings over a field having elements of infinite multiplicative order (\Cref{theorem-big-interpretation}); such subclass clearly satisfies $\CF_{4} \subseteq \CF_{2} \cup \CF_{3}$, and overlaps with both $\CF_{2}$ and $\CF_{3}$ (\Cref{proposition-good-fields}).}

\begin{theorem} \label{theorem-interpretability-R[X]-reduced-indecomposable}

Arithmetic is interpretable in reduced indecomposable polynomial rings \textup{(}uniformly in certain subclasses whose union covers the entire class\textup{)}.

\end{theorem}
\vspace{3mm}
\noindent Since these interpretations are not uniform, they do not provide a uniform translation of sentences about arithmetic to sentences about the members of the class $\CF$ of reduced indecomposable polynomial rings. However, such a uniform translation exists, by invoking~\Cref{proposition-auxiliary-structures} together with the following result, which constitutes a nontrivial extension of the scope of a method due to \noun{Raphael Robinson} \cite{RobR51}*{\S\S 4b,4c}:

\begin{theorem}[\cite{BCN-main}*{Theorem 6.1}] \label{theorem-undecidability-R[X]-reduced-indecomposable}

Let \( \CL = ( + , \times ) \) be the standard signature of rings. There is a uniform translation of sentences about the structure \( ( \Zp , + , \mid \, ) \) to \( \CL \)-sentences about the members in \( \CF \).

\end{theorem}
 
\subsection{Achieving (uniform) interpretability of arithmetic and/or undecidability through factor rings} \label{subsection-from-S/I-to-S}

\noindent Let $S$ be a ring, and let $\theta$ be a one-variable formula in the language of rings. If $\theta$ defines an ideal $I$ in $S$, then there is a canonical one-dimensional interpretation $\Sigma = \Sigma_{ S , \theta }$ of the factor ring $S / I$ in the ring $S$ (see~\Cref{definition-interpretation}): the domain formula is given by the tautology $\partial_{ \Sigma }(t) \colon t = t$, so that $\Delta_{ \Sigma } = S$. The coordinate map $F_{ \Sigma } \colon S \to S / I$ is the natural projection, that is $F_{ \Sigma }(s) = s + I$, for each $s \in S$. Finally, the formula for interpretation of equality is given by
\vspace{2mm}
\begin{alignat*}{2}
           b =_{ \Sigma } c\ & \colon \ \ \theta ( b - c ) \quad                       &     & ( \equiv b - c \in I ) \, , \\[1mm]
\shortintertext{ and the formulas for interpretation of sum and product are given by } \\[-5mm]
      [ b + c = d ]_{\Sigma} & \colon \ \ \theta \bigl( d - ( b + c ) \bigr) \quad     &     & ( \equiv d - (b + c) \in I ) \, ; \\[1mm]
[ b \cdot c = d ]_{ \Sigma } & \colon \ \ \theta ( d - bc ) \quad                      &     & ( \equiv d - bc \in I ) \, .
\end{alignat*}
\vspace{1mm}
\noindent Let $\CH$ be a class of rings, and suppose that $\theta$ defines, in each ring $S \in \CH$, an ideal $I_{S}$. Since the formulas above depend only on $\theta$ and not on $S \in \CH$, the same happens to each one of the (formulas of the) interpretations $\Sigma_{ S , \theta }$, with $S \in \CH$. Consequently, all the induced translations $\stixwidehat{ \Sigma_{ S , \theta } }$ of sentences (see~\Cref{definition-induced-translation}) from the factor rings $S / I_{S}$ to sentences about the corresponding rings $S$ also coincide. We denote the coincident interpretations $\Sigma_{ S , \theta }$ by $\Sigma_{ \theta }$, and the coincident translations $\stixwidehat{ \Sigma_{ S , \theta } }$ by $\stixwidehat{ \Sigma_{ \theta } }$.

Let $\CF$ be the class of factor rings $S / I_{S}$, with $S \in \CH$. If a uniform interpretation $\Gamma$ of arithmetic in $\CF$ is given, then by composing $\Gamma$ with the coincident interpretation $\Sigma_{ \theta }$ above we obtain a uniform interpretation of arithmetic in $\CH$.

Similarly, a given uniform translation of sentences about arithmetic to sentences about members of $\CF$ can be composed with the coincident translations $\stixwidehat{ \Sigma_{ \theta } }$ (of sentences about the factor rings $S / I_{S}$ to sentences about the corresponding rings $S$), providing in this way a uniform translation of sentences about arithmetic to sentences about the members of $\CH$ (and consequently a uniform proof of undecidability for all members of $\CH$).

\subsection{Removing the reducedness hypothesis}

\noindent We apply the results of~\Cref{subsection-from-S/I-to-S} in the following context: given a ring $S$, let $N_{S}$ be the \textbf{nilradical} of $S$, that is, the set of nilpotent elements in $S$. It is well known that $N_{S}$ is actually an ideal, and clearly the factor ring $S / N_{S}$ is reduced. We prove below (\Cref{lemma-basic-nilradical-c} and~\Cref{proposition-nilradical-polynomial-a}) that $S$ being indecomposable (resp.~polynomial) implies that $S / N_{S}$ is indecomposable (resp.~polynomial), and that $N_{S}$ is uniformly definable in the whole class of polynomial rings (\Cref{proposition-nilradical-polynomial-c}). We point out that the nilradical of a ring is not always first-order definable; see~\cite{Hod93}*{Exercise 8.5.1} for an example\footnotemark\!\!.

\footnotetext{One might expect that, given the infinitary definition of the nilradical, there would be plenty of examples of rings lacking a definable nilradical. Surprisingly, this is not the case, and the construction of the cited textbook counterexample seems to involve some form of choice, so perhaps it cannot be carried out directly in \textsf{ZF}.}

Thus, if $\CH$ is the class of indecomposable polynomial rings, then there is a formula $\theta$ that defines $N_{S}$ for each $S \in \CH$, and the class $\CF$ of factor rings $S / N_{S}$ is precisely the class of reduced indecomposable polynomial rings. These facts, together with~\Cref{theorem-interpretability-R[X]-reduced-indecomposable}, imply that arithmetic is interpretable in indecomposable polynomial rings (uniformly in certain subclasses whose union covers the entire class). Similarly, these facts together with~\Cref{theorem-undecidability-R[X]-reduced-indecomposable} provide a uniform proof of undecidability for all indecomposable polynomial rings.

Before we proceed with the proofs of the pending results, we set the usual notation $\overline{b}$ for the image of $b \in S$ under the natural projection $S \to S / I$ (with $I$ and ideal in $S$).

\pagebreak

\begin{lemma}

Let \( S \) be a ring.

\begin{enumlemma}

\item \label{lemma-basic-nilradical-a} The sum of a unit and a nilpotent element is a unit.

\item \label{lemma-basic-nilradical-b} If \( f \in S \) is regular and idempotent, then \( f = 1 \).

\item \label{lemma-basic-nilradical-c} If \( S \) is indecomposable, then \( S / N_{S} \) is too.

\end{enumlemma}

\end{lemma}

\begin{proof} \leavevmode

\begin{enumerate}

\item Let $u \in S^{*}$ and $s \in N_{S}$. We have $s^{k} = 0$ for some $k \geq 1$, hence $u^{k} - ( -s )^{k} = u^{k}$ is a unit. Since $u + s = u - ( -s )$ divides $u^{k} - ( -s )^{k}$, the result follows.

\item Immediate from the equality $f^{2} = f$ and the regularity of $f$.

\item Suppose $S$ indecomposable, and let $e \in S$ be such that $( 1 - e ) e \in N_{S}$. Assume at this point that $e \in S^{*}$ or $1 - e \in S^{*}$. Therefore one of the elements $\overline{e}$ or $\overline{ 1 - e }$ is invertible and idempotent in $S / N_{S}$, so by item~\subcref{lemma-basic-nilradical-b} we have $\overline{e} = \overline{1}$ or $\overline{ 1 - e } = \overline{1}$, which is the desired result. 

In order to finish the proof, suppose $( 1 - e ) e \in N_{S}$, say $( 1 - e )^{k} e^{k} = 0$, with $k \geq 1$. Writing $( 1 - e )^{k} = 1 - ce$, with $c \in S$, and defining $d = \sum_{ i = 0 }^{ k - 1 } ( ce )^{i}$, we get $( 1 - e )^{k} d = ( 1 - ce ) d = 1 - (ce)^{k}$, hence
\[
( ce )^{k} [ 1 - ( ce )^{k} ] = c^{k} e^{k} [ ( 1 - e )^{k} d ] = c^{k} d [ ( 1 - e )^{k} e^{k} ] = 0.
\]
Thus, $( ce )^{k}$ is idempotent, hence it equals to $1$ or $0$ because $S$ is indecomposable. If $( ce )^{k} = 1$, then $e$ is invertible. Otherwise $ce \in N_{S}$, so trivially $-ce \in N_{S}$, and therefore $( 1 - e )^{k} = 1 + ( -ce )$ is a unit by item~\subcref{lemma-basic-nilradical-a}, so $1 - e$ is a unit in this case.\qedhere

\end{enumerate}

\end{proof}

\noindent For a set $\CX$ of indeterminates over a ring $R$ and an ideal $I$ in $R$, the universal property for polynomial rings guarantees the existence of a (unique) ring homomorphism $R[ \CX ] \to ( R / I ) [ \CX ]$ acting as the identity on $\CX$ and sending each element $r \in R$ to the element $r + I \in R / I$. This mapping is surjective and has kernel equal to the set $I[ \CX ]$ of polynomials with coefficients in $I$, so it induces an isomorphism between $R[ \CX ] / I[ \CX ]$ and $( R / I ) [ \CX ]$.

\begin{proposition}

Let \( R \) be a ring, and let \( S = R[ \CX ] \), with \( \CX \) a set of indeterminates over \( R \).

\begin{enumproposition}

\item \label{proposition-nilradical-polynomial-a} We have \( N_{S} = N_{R} [ \CX ] \), and therefore \( S / N_{S} \) is isomorphic to \( ( R / N_{R} ) [ \CX ] \).

\item \label{proposition-nilradical-polynomial-b} The units in \( S \) are precisely the polynomials with constant coefficient invertible in \( R \) and all the other coefficients nilpotent.

\item \label{proposition-nilradical-polynomial-c} \( N_{S} \) is precisely the set \( \{ f \in S \colon 1 + fg \in S^{*} \ \textnormal{for all} \ g \in S \} \). Consequently, \( N_{S} \) is uniformly definable in the whole class of polynomial rings.

\end{enumproposition}

\end{proposition}

\begin{proof}

Since $N_{R} \subseteq N_{S}$, we always have $N_{R} [ \CX ] \subseteq N_{S}$; in particular, if $f = r + g$, with $r \in R^{*}$ and $g \in N_{R}[ \CX ]$, then $g \in N_{S}$, and so $f \in S^{*}$ by~\Cref{lemma-basic-nilradical-a}. Moreover, if $f \in N_{S}$, then for all $g \in S$ we have $fg \in N_{S}$, hence $1 + fg \in S^{*}$ (again by~\Cref{lemma-basic-nilradical-a}). Thus, the trivial inclusion in each one of items~\subcref{proposition-nilradical-polynomial-a}--\subcref{proposition-nilradical-polynomial-c} is proved.

Assume that items~\subcref{proposition-nilradical-polynomial-a}--\subcref{proposition-nilradical-polynomial-c} hold true in the univariate case. If $x_{1} , \dotsc , x_{n}$ are indeterminates over $R$, then $x_{n}$ is an indeterminate over $R' \coloneq R[ x_{1} , \dotsc , x_{ n - 1 } ]$ and $R[ x_{1} , \dotsc , x_{n} ]$ is isomorphic to $R' [ x_{n} ]$. Therefore, by applying induction on $n$ we prove items~\subcref{proposition-nilradical-polynomial-a} and~\subcref{proposition-nilradical-polynomial-b} whenever $\CX$ is finite.

In the general case, let $f \in S$ and let $S' = R [ \CX' ]$, with $\CX'$ being a finite subset of $\CX$ such that $f \in S'$. If $f \in N_{S}$, then $f \in N_{ S' } = N_{R} [ \CX' ] \subseteq N_{R} [ \CX ]$, which proves item~\subcref{proposition-nilradical-polynomial-a}. Similarly, If $f \in S^{*}$, then after enlarging $\CX'$ we may assume that $f , f^{ -1 } \in S'$, hence $f \in (S')^{*}$, and so from the finitely many variables case we conclude that $f$ has constant invertible coefficient and all the other coefficients nilpotent. On the other hand, the result of item~\subcref{proposition-nilradical-polynomial-c} follows directly after writing $S$ as a univariate polynomial ring (without the need of the intermediate reasoning used in the proof of items~\subcref{proposition-nilradical-polynomial-a} and~\subcref{proposition-nilradical-polynomial-b}).

Thus, it remains to prove the nontrivial inclusion in each one of items~\subcref{proposition-nilradical-polynomial-a}--\subcref{proposition-nilradical-polynomial-c} just in the univariate case. Let $S = R[x]$, with $x$ an indeterminate over $R$.

\begin{enumerate}

\item If $f = f_{d} x^{d} + \dotsb + f_{0} \in N_{S}$, with $d \geq 1$, then $f^{k} = 0$ for some $k \geq 1$. In particular the coefficient of $x^{ dk }$ in $f^{k}$, namely $f_{d}^{k}$, must be zero and so $f_{d} \in N_{R}$. Therefore $f - f_{d} x^{d} \in N_{S}$, and so we may iterate this argument to obtain that $f_{i} \in N_{R}$ for all $i$.

\item Let $f = f_{d} x^{d} + \dotsb + f_{0} \in S^{*}$, with $d \geq 1$, and let $g = g_{m} x^{m} + \dotsb + g_{0} \in S$ with $m \geq 1$ such that $fg = 1$. Recall that $fg = \sum_{ k = 0 }^{ d + m } c_{k} x^{k}$, where

\[
c_{k} = \sum_{ \substack{ 0 \leq i \leq d \\ 0 \leq j \leq m \\ i + j = k } } f_{i} g_{j} \, .
\]
In particular we have $1 = c_{0} = f_{0} g_{0}$, so both $f_{0}$ and $g_{0}$ are units in $R$, and $c_{k} = 0$ for $k = 1 , \dotsc , d + m$.

We claim that, for $r = 1 , \dotsc , m + 1$, multiplying by $f_{d}^{r}$ annihilates the coefficients $g_{ m - r + 1} , g_{ m - r + 2 } , \dotsc , g_{m}$. For $r = 1$ the claim follows from the equality $f_{d} g_{m} = c_{ d + m } = 0$ (because $0 < d + m$). Suppose that the claim holds for a value $r$ with $1 \leq r \leq m$, so that $0 < d + m - r$ (because $d \geq 1$ and $m - r \geq 0$), and therefore
\begin{align*}
0 & = c_{ d + m - r } \\[2mm]
  & = f_{d} g_{ m - r } + \Biggl( \, \sum_{ t = 1 }^{ \min \{ d , r \} } f_{ d - t } g_{ m - r + t } \Biggr) \\[2mm]
\shortintertext{ (recall that we must have $d - t \geq 0$ and $m - r + t \leq m$) } \\[-3mm]
  & = f_{d} g_{ m - r } + A \, .
\end{align*}
\vspace{1mm}
\noindent By induction hypothesis $f_{d}^{r}$ annihilates each coefficient of $g$ appearing in the summands of $A$ (they correspond to degrees $m - r + t \geq m - r +1$). Therefore we have $f_{d}^{r} A = 0$, and so by multiplying the equality above by $f_{d}^{r}$ one gets $0 = f_{d}^{ r + 1 } g_{ m - r } = f_{d}^{ r + 1 } g_{ m - ( r + 1 ) + 1 }$, and since $f_{d}^{ r + 1 }$ also annihilates $g_{ m - r + 1} , g_{ m - r + 2 } , \dotsc , g_{m}$, this completes the proof by induction.

Finally, by taking $r = m + 1$ we get $0 = f_{d}^{ r + 1 } g_{ m - ( m + 1 ) + 1 } = f_{d}^{ r + 1 } g_{0}$. As $g_{0}$ is a unit in $R$, we conclude that $f_{d}^{ r + 1 } = 0$, hence $f_{d} \in N_{R}$. Thus, we have $f_{d} x^{d} \in N_{S}$, and~\Cref{lemma-basic-nilradical-a} implies then that $f - f_{d} x^{d} \in S^{*}$. Iterating this argument we conclude that $f_{i} \in N_{R}$ for $i = 1 , \dotsc , d$.

\item If $f \in S$ satisfies $1 + fg \in S^{*}$ for all $g \in S$, then in particular $1 + xf \in S^{*}$, so by item~\subcref{proposition-nilradical-polynomial-b} all the coefficients of $1 + xf$ of positive degree are nilpotent. Since these coefficients are precisely the coefficients of $f$, we conclude that $f \in N_{R} [x] = N_{S}$ (by item~\subcref{proposition-nilradical-polynomial-a}).\qedhere

\end{enumerate}

\end{proof}

\begin{remark}

The set $J_{S} = \{ f \in S \colon 1 + fg \in S^{*} \ \textnormal{for all} \ g \in S \}$ appearing in the statement of~\Cref{proposition-nilradical-polynomial-c} is known, in the literature, as the \textbf{\noun{Jacobson} radical} (of the ring $S$), and it is equal to the intersection of all the maximal ideals in $S$, provided every nonunit in $S$ belongs to at least one such ideal. As the reader may check, our proof of~\Cref{proposition-nilradical-polynomial-c} does not rely at all on such characterization, and in particular we do not need to resort to these ideals (whose existence, in arbitrary nonzero commutative unital rings, is known to be equivalent to the axiom of choice~\cite{Hod79}). Similarly we do not use, in any part of our work, the characterization of the nilradical of a ring as the intersection of its prime ideals, which is known to be equivalent to \textsf{BPI}, the Boolean prime ideal theorem (this result is independent of \textsf{ZF}; see~\cite{HR98}*{Forms 14 AL and 14 AN}).

\end{remark}

\end{document}